\documentclass[12pt]{elsarticle}

\usepackage{amsthm,amsmath,amssymb}

\usepackage{graphicx}

\usepackage[colorlinks=true,citecolor=black,linkcolor=black,urlcolor=blue]{hyperref}

\theoremstyle{plain}
\newtheorem{theorem}{Theorem}
\newtheorem{lemma}[theorem]{Lemma}
\newtheorem{corollary}[theorem]{Corollary}
\newtheorem{proposition}[theorem]{Proposition}

\newtheorem{claim}[theorem]{Claim}

\theoremstyle{definition}
\newtheorem{definition}[theorem]{Definition}

\newtheorem{conjecture}[theorem]{Conjecture}

\theoremstyle{remark}
\newtheorem{remark}[theorem]{Remark}

\title{\bf Proof of a conjecture of Bauer, Fan and Veldman}

\author{Tri Lai\\
\small Department of Mathematics\\[-0.8ex]
\small Indiana University\\[-0.8ex]
\small Bloomington, IN, U.S.A.\\
\small\tt tmlai@indiana.edu
}

\date{
}

\begin{document}


\begin{abstract}
For a 1-tough graph $G$ we define $\sigma_3(G) = \min\{\deg(u) + \deg(v)+
\deg(w):$ $\{u, v, w\}$ is an independent set of vertices$\}$ and $NC2(G)=\min \{|N(u)\cup N(v)|: d(u,v)=2\}$.
D. Bauer, G. Fan and H.J.Veldman proved that $c(G)\geq \min\{n,2NC2(G)\}$ for any 1-tough graph $G$ with $\sigma_3(G)\geq n\geq 3$, where $c(G)$ is the circumference of $G$ (D. Bauer, G. Fan and H.J.Veldman,
 \emph{Hamiltonian properties of graphs with large neighborhood unions},
Discrete Mathematics, 1991). They also conjectured a stronger upper bound for the circumference:
$c(G)\geq
\min\{n,2NC2(G)+4\}$.
 In this paper, we prove this conjecture.

\end{abstract}
\maketitle
\section{Introduction.}
We consider only finite undirected graphs without loops and multiple edges. A graph G is said to be \emph{k-connected} if there does not exist a set of $k-1$ vertices whose removal disconnects the graph. A graph $G$ is \emph{1-tough} if for every nonempty $S\subset V(G)$ the graph $G-S$ has at most $|S|$ components. All paths and cycles in this paper are simple paths and simple cycles, respectively. A set of vertices of the graph $G$ is \emph{independent} if no two of its elements are adjacent. The \textit{independence number} of $G$, denoted by $\alpha(G)$, is defined by setting $\alpha(G)=\max\{|U|:\ U\subseteq V(G) \text{ independent}\}$. A cycle $C$ of $G$ is called  a \emph{dominating cycle} if the vertices of graph $G-C$ is independent.  A graph which contains a dominating cycle is called a \emph{dominating graph}.

 The \emph{smallest union degree of order $3$}  of $G$, denoted by $\sigma_3(G)$,  is defined by setting
\begin{equation}
\sigma_3(G):=
\begin{cases}
\min \{\sum_{i=1}^3\deg(v_i): &\text{$\{v_1,v_2,v_3\}$ is independent if $\alpha(G)\geq 3$;}\\
3(n-1), &\text{ otherwise,}
\end{cases}
\end{equation}
where $\deg(v)$ is the degree of vertex $v$ in $G$. Denote by $N(v)$ the set vertices adjacent to $v$, called the neighbor set of $v$. We denote $N(x,y):=N(x)\cup N(y)$, for any two distinct vertices $x$ and $y$. Next, we define
\begin{equation}
NC2(G):=
\begin{cases}
\min\{| N(u)\cup N(v)|:\ d(u,v)=2\}, &\text{ if $G$ is not complete;}\\
 n-1, &\text{otherwise,}
\end{cases}
\end{equation}
where $d(u,v)$ is the distance between $u$ and $v$.

Denote by $c(G)$ the circumference of $G$, i.e. the length of the longest cycle in $G$. We have the following lower bound for the circumference of a 1-tough graph due to Bauer, Fan and Veldman.

\begin{theorem}[\cite{Bau91}, Theorem 26]\label{Bau}
If $G$ is 1-tough and $\sigma_3(G)\geq n\geq 3$, then
\begin{equation}
c(G)\geq \min\{n,2NC2(G)\}.
\end{equation}
\end{theorem}

We have a better lower bound on $c(G)$ due to Hoa in the theorem below.
\begin{theorem}[\cite{Hoa98}, Theorem 1]\label{HOA}
 If G is 1-tough graph and $\sigma_3(G)\geq n\geq 3$, then there exists an independent set of $\sigma_3-n+5$ elements $\{v_0,v_1,\ldots, v_{ \sigma_3(G)-n+4}\}$ such that the distance between $v_0$ and $v_i$ equals 2, for any $1\leq i \leq \sigma_3(G)-n+4$, and
 \begin{equation}
 c(G) \geq \min\left\{n,2\left| \bigcup_{i=0}^{\sigma_3(G)-n+4} N(v_i)\right|+2\right\}.
 \end{equation}
 \end{theorem}
One readily sees that  Hoa's result  implies
\begin{equation}\label{HOAeq}
c(G)\geq \min\{n,2NC2(G)+2\}.
\end{equation}
Bauer, Fan and Veldman also conjectured in \cite{Bau91} that
\begin{conjecture}[\cite{Bau91}, Conjecture 27]\label{Bauconj}
Assume that $G$ is a 1-tough graph with $\sigma_3(G)\geq n\geq 3$. Then
\begin{equation}\label{conjeq}
c(G)\geq \min\{n,2NC2(G)+4\}.
\end{equation}
\end{conjecture}

Our goal in the paper is to prove this conjecture. We divide the proof of the conjecture into two steps: we show first that $c(G)\not=2NC2(G)+3$, then show that $c(G)\not= 2NC2(G)+2$ (then the Conjecture \ref{Bauconj} follows (\ref{HOAeq})).

 We note that, in \cite{Hoa95} and \cite{Hoa98}, Hoa defined $\sigma_3(G):=k(n-\alpha(G))$ or $\infty$ when $\alpha(G)\leq 2$ (as opposed to being $k(n-1)$ in our definition above). However, in all definitions we always have the key condition $\sigma_3(G)\geq n$ when $\alpha(G)\leq 2$ and $n\geq 3$. This guarantees that we can use the results in \cite{Hoa95} and \cite{Hoa98} without redefining $\sigma_3(G)$. We have the same situation in the definition of $NC2(G)$. Hoa defined $NC2(G)=n-\alpha(G)$ when $G$ is the complete graph $K_n$ (as opposed to being $n-1$ in our definition). However, in this case, we have $c(G)=n$, then the Conjecture 3 is obviously true. Again, we do not need to worry about this minor difference in the definitions of $NC2(G)$.

\section{Preliminaries}

Let $C$ be a longest cycle in the graph $G$ on $n\geq3$ vertices. Denote by $\overrightarrow{C}$ the cycle $C$ with a given orientation. Denote by $x^-$ and $x^+$ the
\emph{predecessor} and \emph{successor} of $x$ on $\overrightarrow{C}$, respectively. Further define, $x^{+i}:=(x^{+(i-1)})^+$ and
$x^{-i}:=(x^{-(i-1)})^-$, for $i\geq 2$.
If $A\subseteq V(C)$, we dfine two sets $A^+:=\{x : x ^-\in A\}$, and $A^-:=\{x: x^+ \in A\}$.

 If $u$ and $v$ are on the cycle $C$, then $u\overrightarrow{C}v$ denotes the set of consecutive vertices on $C$ from $u$ to $v$ in the direction specified by $\overrightarrow{C}$. The same vertices, in reverse order, are given by $v\overleftarrow{C}u$. We consider $u\overrightarrow{C}v$ and $v\overleftarrow{C}u$ both as paths and as vertex sets, and we call them \textit{$C$-paths}. 
In this paper, we will use the notations system in Diestel \cite{Diestel} to represent paths and cycles.

If $G$ is a non-hamiltonian graph, then we define
\begin{equation}
\mu(C):=\max \{deg(v):v\in V(G-C)\},
\end{equation}
for every cycle $C$ in $G$, and
\begin{equation}
\mu(G):= \max \{ \mu(C): C \mbox{ is a longest cycle in } G\}.
\end{equation}

\begin{lemma}[\cite{Bau90}, Theorem 5]\label{l1}
Let $G$  be a 1-tough graph on $n$ vertices such that $\sigma_3(G)\geq n$. Then every longest cycle in $G$ is a dominating cycle.
\end{lemma}



\begin{lemma}[\cite{Hoa95}, Lemma 2]\label{l2}
Let $G$  be a 1-tough graph on $n$ vertices with $\sigma_3(G)\geq n\geq3$. If $G$ is non-hamiltonian, then
$G$ has a longest cycle $C$ such that
$\mu(C) \geq \frac{1}{3}n.$
\end{lemma}

 We get the following lemma by modifying the proof of Theorem 26 in \cite{Bau91}
\begin{lemma}\label{l3}
Let $G$ be a non-hamiltonian 1-tough graph on $n$ vertices with  $\sigma_3(G)\geq n\geq3$. We can find
a longest cycle $\overrightarrow{C}$, a vertex $u\notin V(C)$, and a vertex $v$ on $\overrightarrow{C}$, such that
$v^+,v^- \in N(u)$. Moreover, for $B:=N(u)\cup N(v)$ we have $B\subseteq V(C)$  and $B\cap B^+=B\cap B^-=\emptyset$.
\end{lemma}

\begin{proof}
Follow the argument in the proof of Theorem 26 in \cite{Bau91}. The only difference is that the condition ``$G$ is 2-connected and $\sigma_3\geq n+2$" is  replaced
by ``$G$ is 1-tough and $\sigma_3\geq n$", and we use  Lemma \ref{l2} above instead of Lemma 22 in [2].
\end{proof}

Let $G$ be a graph satisfying the hypothesis in Lemma \ref{l3}. Assume that $B:=N(u)\cup N(v)=\{b_1,b_2,\dotsc,b_m\}$, where the vertices $v^{+}\equiv b_1,b_2,\dotsc,b_m \equiv v^-$ appear successively on $\overrightarrow{C}$. A $C$-path connecting two successive vertices of set $B$, i.e. having the form $b_i\overrightarrow{C}b_{i+1}$, is called an \emph{interval}. An interval consisting of $k$ edges is called a \emph{$k$-interval}, for $k\geq 2$ (by Lemma \ref{l3} there no ``1-interval"). Let $\mathcal{P}$ be a path in $G$. A vertex $x\in\mathcal{P}$ is called \emph{an inner vertex} of $\mathcal{P}$ if
$x$ is different from the ends of $\mathcal{P}$. We say that two $C$-paths $\mathcal{P}$ and $\mathcal{P}'$
are \emph{inner-connected}  if  some
inner vertex $x$ of $\mathcal{P}$  is incident to some inner vertex $y$ of
$\mathcal{P}'$. If $\mathcal{P}$ and $\mathcal{P}'$ are not inner-connected, then we say they are \emph{inner-disconnected}.

\begin{lemma}[\cite{Hoa98}, Lemma 4]\label{l5}
Let $G$ be a non-hamiltonian 1-tough graph on $n$ vertices with  $\sigma_3(G)\geq n\geq3$.  Let $C$ be a longest cycle in $G$, $u$ a vertex not in $C$, and $v$ a vertex in $C$ so that $v^+,v^-\in N(u)$.  Assume in addition that $A$ is the set of all inner vertices of 2-intervals whose end points are in $N(u)$. Then $V(G-C)\cup N(u)^+\cup N(A)^+$ and $V(G-C)\cup N(u)^-\cup N(A)^-$ are two independent sets.
\end{lemma}
We have two corollaries of Lemma \ref{l5} as follows.
\begin{corollary}\label{l5c}
If $G$ satisfies the hypothesis of Lemma \ref{l5} and $B:=N(u)\cup N(v)$, then
$B^+\cup V(G-C)$ and $B^-\cup V(G-C)$ are two independent sets.
\end{corollary}
\begin{proof}
Note that $v\in A$, so we have
 \begin{align}
 B^+\cup V(G-C)&=N(u)^+\cup N(v)^+\cup V(G-C)\notag\\
 &\subseteq V(G-C)\cup N(u)^+\cup N(A)^+.
 \end{align}
By Lemma \ref{l5}, $V(G-C)\cup N(u)^+\cup N(A)^+$ is independent, so $B^+\cup V(G-C)$ is also independent. Analogously, $B^-\cup V(G-C)$ is independent.
\end{proof}

\begin{corollary}\label{rmk3}
If $G$ satisfies the hypothesis of Lemma \ref{l5}, a $k$-interval, for $k=2$ or $3$, does not inner-connect to any other 2-intervals in $G$.
\end{corollary}
\begin{proof}
An inner vertices $x$ in the $k$-interval is in $B^+\cup B^-$, for $k=2,3$, where $B:=N(u)\cup N(v)$. Thus by Corollary \ref{l5c}, $x$ is not adjacent to any inner vertex $y$ of a 2-interval, since $y$ is in $B^+\cap B^-$.
\end{proof}

The following lemma was also proved in \cite{Hoa98}.
\begin{lemma}[\cite{Hoa98}, Lemma 9]\label{l4}
Assume that $G$ is a non-hamiltonian 1-tough  graph on $n$ vertices
with $\sigma_3(G)\geq n\geq 3$. Then $G$ contains a longest cycle $C$
avoiding a vertex $u$ with $deg(u)=\mu(G)$ and $s\geq
\sigma_3(G)-n+4$, where $s$ is the number of $2$-intervals whose end points are in $N(u)$.
\end{lemma}

Note that if $G$ is hamiltonian, then the inequality (\ref{conjeq}) in Conjecture \ref{Bauconj} is obviously true. Therefore, we only need to consider the case $G$ is non-hamiltonian. Follow the light of Lemmas \ref{l1}, \ref{l2}, \ref{l3},\ref{l5}, and \ref{l4}, we assume from now on a setup \textbf{(S)} as follows.
\bigskip

\textbf{SETUP (S)}
\begin{enumerate}
\item[($S_1$)]\emph{$G$ is a non-hamiltonian 1-tough  graph on $n$ vertices with $\sigma_3(G)\geq n\geq 3$.}
 \item[($S_2$)] \emph{$\overrightarrow{C}$ is a longest cycle of $G$,  $u$ is a vertex not in $V(C)$, and $v$ is a vertex in $V(C)$ such that $v^+,v^- \in N(u)$. Let $B:=N(u)\cup N(v)=\{b_1,b_2,\dotsc,b_m\}$, where the vertices $v^{+}\equiv b_1,b_2,\dotsc,b_m \equiv v^-$ appear successively on $\overrightarrow{C}$. }
 \item[($S_3$)]\emph{There at least four 2-intervals whose ends are all in $N(u)$ or are all in $N(v)$.}
\end{enumerate}

\bigskip

Given a graph $G$ satisfying assumption $(S_1)$, then a setup \textbf{(S)} in $G$  is determined uniquely by a vertices-cycle triple $(u,v,\overrightarrow{C})$.


\begin{remark}[\textbf{Reversing Orientation Trick}]\label{trick1}
 If we reverse the orientation of $C$,  then the set $B^-$ on $\overrightarrow{C}$ is now the set $B^+$ on $\overleftarrow{C}$. In the proof of Corollary \ref{l5c}, the independence of the set $B^{-}\cup V(G-C)$ follows the independence of $B^{+}\cup V(G-C)$ on the reverse-orientation of $C$. Reversing the orientation (of $C$) is a useful trick that will be used frequently in this paper.
\end{remark}

\begin{lemma}\label{l7}
(Assume a setup \textbf{(S)}) Assume that all intervals on $\overrightarrow{C}$ are pairwise inner-disconnected. Then there are at least two intervals of length greater than 3.
\end{lemma}

\begin{proof}
Recall from the setup \textbf{(S)} that $B:=\{b_{1},\dots,b_{m}\}$, for some positive integer $m$.

Assume otherwise that there is at most one interval of length greater than 3. Then other intervals have length at most 3, so their inner vertices are in $B^+\cup B^-$. We will show that $G$ is not a 1-tough graph.

Consider the graph $G-B$. We have two facts stated below.

(1) Two vertices $b_{i}^{+}$ and $b_{j}^{+}$ are not in the same component of $G-B$, for any $1\leq  i\not=j\leq m$.

 Indeed, assume otherwise that we can find a path $\mathcal{P}=v_{1}v_2\dots v_{t}$ in $G-B$ connecting $b_{i}^{+}$ and $b_{j}^{+}$, i.e. $v_1\equiv b_{i}^{+}$  and $v_t\equiv b_{j}^{+}$. Since we have at most one interval of length greater than 3, we can assume that $b_j^+$ is in an interval of length at most 3. By Corollary \ref{l5c} and the inner-disconnectedness of the intervals, $v_{t-1}$ must be an inner vertex of the interval $b_{j}\overrightarrow{C}b_{j+1}$ (besides the inner vertex $v_t\equiv b_j^+$). Thus, $b_{j}\overrightarrow{C}b_{j+1}$ must have length 3, and $v_{t-1}\equiv b_j^{+2}\equiv b_{j+1}^-\in B^-$. By the Corollary \ref{l5c} and and the inner-disconnectedness of the intervals again, $v_{t-2}$ in turn must be an inner vertex of the interval $b_{j}\overrightarrow{C}b_{j+1}$. However, this implies that $b_{j}\overrightarrow{C}b_{j+1}$ must have length at least 4, a contradiction.

(2) Two vertices $b_{i}^{+}$ and $u$ are not in the same component of $G-B$, for any $i=1,2,\dotsc,m$.

 Indeed, assume otherwise that there is a path $\mathcal{P}=v_{1}\dots v_{t}$ in $G-B$ connecting them, i.e. $v_1\equiv b_{i}^{+}$  and $v_t\equiv u$. Then $v_{t-1}$ is in $N(u)\subseteq B$, a contradiction to the fact that $v_{t-1}$ is a vertex in $G-B$.

From (1) and (2), the graph $G-B$ has at least $m+1$ distinct components, so that each of them contains at most one vertex in the set $\{u\} \cup B^+$. This implies that $G$ is not 1-tough, which contradicts our setup \textbf{(S)}.


\end{proof}

We have two new definitions stated below.

\begin{definition} A pair of vertices $(x,y)$ in graph $G$ is called a \emph{small pair}
if $| N(x,y) |\leq |B| -1$ and $d(x,y)=2$.
\end{definition}

\begin{definition}
Assume $b_i$ is a  vertex in $B:=N(u)\cup N(v)$ such that $1<i<m$, i.e. $b_i \not= v^+,v^-$. A path
 $\mathcal{P}$ in $G$ is called a \emph{bad path} if it has one of
the following two forms:

(i) $\mathcal{P}$ consists of all vertices in  $v^+\overrightarrow{C}b_i^-$, and the ends of $\mathcal{P}$
are $v^+$ and some vertex $b_j\in B$, for $1\leq j<i$.

(ii) $\mathcal{P}$ consists of all vertices in $b_i^+\overrightarrow{C}v^-$, and the ends of $\mathcal{P}$
are $v^-$ and  some $b_j \in B$, for $i<j\leq m$.
\end{definition}


We have two key results about small pairs and bad paths shown below.
\begin{proposition}\label{small}
If $|B| =NC2(G)$, then there are no small pairs.
\end{proposition}
\begin{proof}
Assume otherwise that $|B| =NC2(G)$ and $(x,y)$ is a small pair. Then $|N(x,y)|\leq |B|-1$. By definition of $NC2(G)$, we have $|B|=NC2(G)\leq |N(x,y)|\leq |B|-1$,
a contradiction.
\end{proof}
\begin{proposition}\label{alpha}
There are no bad paths in $G$.
\end{proposition}
\begin{proof}
Suppose otherwise that there is a bad path $\mathcal{P}$ of form (i). Assume that $v^+$ and $b_j$ are two ends of $\mathcal{P}$. We will construct a cycle $C'$ that is longer than $C$, and then get a contradiction. There are four possible cases as follows:
\begin{enumerate}
\item[(1)] If $b_iu$ and $b_ju \in E(G)$, then let $C':=\  v^+vv^-\overleftarrow{C}b_iub_j\mathcal{P}v^{+}$ (see Figure \ref{Fig1}(a)).

\item[(2)] If $b_iv$ and $b_jv \in E(G)$, then let $C':=\  v^+uv^-\overleftarrow{C}b_ivb_j\mathcal{P}v^{+}$ (see Figure \ref{Fig1}(b)).

\item[(3)] If $b_jv$ and  $b_iu \in E(G)$, then let
$C':=\  v^+ub_i\overrightarrow{C}v^-vb_j{}_{\mathcal{P}}v^{+}$ (see Figure \ref{Fig1}(c)).

\item[(4)] If $b_ju$ and $b_iv \in E(G)$, then let
$C':=\ v^+vb_i\overrightarrow{C}v^-ub_j\mathcal{P}v^{+}$ (see  Figure \ref{Fig1}(d)).
\end{enumerate}
This completes the proof for the case where $\mathcal{P}$ is a bad path of form (i). The case where $\mathcal{P}$ is a bad path of form (ii) follows similarly and is omitted.
~\end{proof}

\begin{figure}\centering
\includegraphics[width=0.80\textwidth]{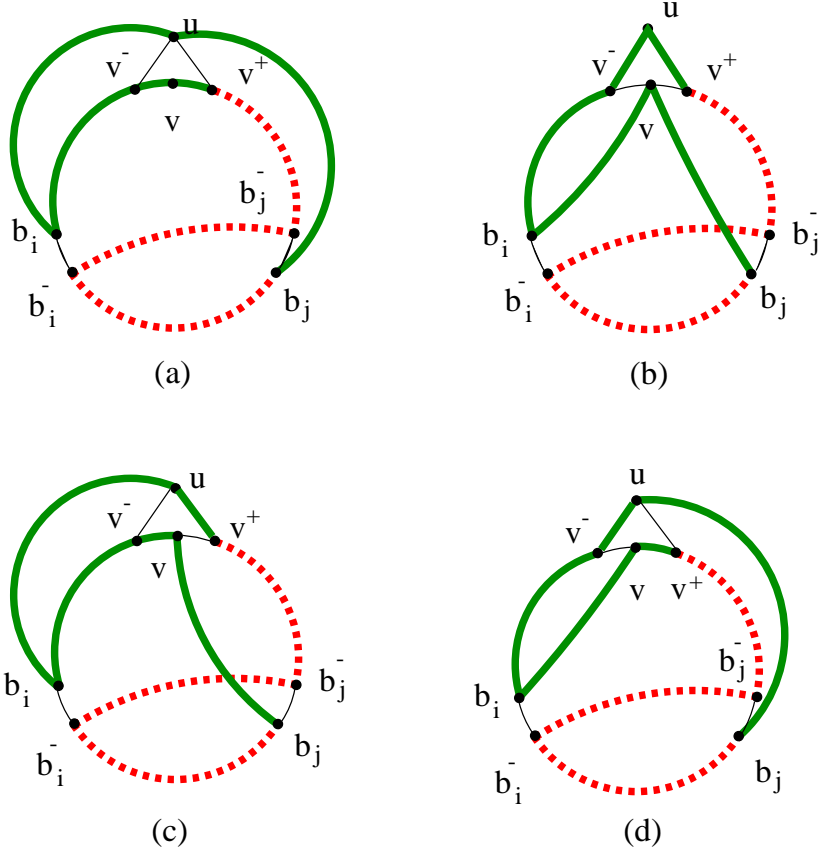}
\caption{Illustrating the proof of Proposition \ref{alpha}. Path $\mathcal{P}$ is the dotted one.}
\label{Fig1}
\end{figure}
\begin{remark}[\textbf{Interchanging roles trick}]\label{trick2}
Consider a new longest cycle $\overrightarrow{\widetilde{C}}:= uv^+\overrightarrow{C}v^-u$ (the orientation of $\widetilde{C}$ follows the order of vertices in its representation). It is easy to see that the triple $(v,u,\overrightarrow{\widetilde{C}})$ determines a new setup \textbf{(S')} in $G$ satisfying three conditions $(S_1)$, $(S_2)$ and $(S_3)$, moreover it has the same set $B$ as \textbf{(S)} does. In particular situations, we need to consider two cases that are the same, except for the roles of $u$ and $v$ are interchanged (for example cases (1)-(2) and cases (3)-(4) in the proof of Propositions \ref{alpha}). Then we only need to consider the first case, the second case is obtained by applying again the argument in the first case to the new setup \textbf{(S')}. Intuitively, the second case in the original setup \textbf{(S)} becomes the first case in the new setup \textbf{(S')}. Thus, for example, in the proof of Proposition \ref{alpha}, we only need to consider two cases (1) and (3), and then the other cases follow naturally from the trick above. We call this trick the \textit{interchanging roles trick}.  Together with the reversing orientation trick in Remark \ref{trick1}, the interchanging roles trick is the main ingredient of our case-by-case proofs.
\end{remark}

We finish this section by quoting a result due to Woodall,  sometimes called Hopping Lemma (see \cite{Boar} and \cite{Wood}).

\begin{lemma}\label{hopping}
Let $\overrightarrow{C}$ be a cycle of length $m$ in a graph $G$. Assume that $G$ contains no cycle of length $m+1$ and no cycle $C'$ of length $m$ with $\omega(G-V(C'))<\omega(G-V(C))$, and $u$ is an isolated vertex of $G-V(G)$. Set $Y_0=\emptyset$ and, for $i\geq i$,
\begin{align}
X_i&=N(Y_{i-1}\cup\{u\}),\\
Y_i&=(X_i\cap V(C))^+\cap (X_i\cap V(C))^-.
\end{align}
Set $X=\bigcup_{i=1}^\infty X_i$, and $Y=\bigcup_{i=1}^{\infty}Y_i$. Then
\begin{enumerate}
\item[(a)] $X\subseteq V(C)$;
\item[(b)] if $x_1,x_2\in X$, then $x_1^+\not=x_2$;
\item[(c)] $X\cap Y=\emptyset$;
\item[(d)] $Y$ is independent.
\end{enumerate}
\end{lemma}
\begin{proof}
Parts (a), (b), (c) were proved by Woodall \cite{Wood}, and part (d) follows from part (c). Indeed, assume otherwise that there are two vertices $x,y\in Y$ so that $xy\in E(G)$. Since $Y=\bigcup_{i=1}^{\infty}Y_i$, we have $x\in Y_i$ and $y\in Y_j$, for some positive integers $i,j$. Then by definition, we have $x\in X_{i+1}$ and $y\in X_{j+1}$, this contradicts part (c).
\end{proof}

We note that in the Woodall's Hopping Lemma \ref{hopping}, $X_1\subseteq X_2 \subseteq\dotsc \subseteq X$ and $Y_1\subseteq Y_2 \subseteq\dotsc \subseteq Y.$

\section{Step 1: Prove $c(G)\not=2NC2(G)+3$}

Assume otherwise that $c(G)=| V(C) |=2NC2(G)+3$. By Lemma $\ref{l3}$, we have
\[|V(C)|\geq | B\cup B^+|= | B| +| B^+|=2| B|.\]
 Moreover, $| B| =| N(u)\cup N(v)|
\geq NC2(G)$, so $| B|$ must be $NC2(G)+1$ or $NC2(G)$.

Note that a $k$-interval contains exactly $k-2$ vertices that are not in $B\cup B^+$. In particular, the $k$-interval $b_i\overrightarrow{C}b_{i+1}$ has $k-2$ vertices $b_i^{+2}$, $b_i^{+3},\dots,b_{i}^{+(k-1)}$ that are not in $B\cup B^+$.

If $| B |= NC2(G)+1$, then $| V(C)| =2| B|+1$. Thus, there is only one vertex in $V(C)- (B\cup B^+)$, say $x$. Clearly, $x$ must be in a $3$-interval ($x^{-2}\overrightarrow{C}x^+$), and all other intervals have length 2. By Corollary \ref{rmk3}, the intervals are pairwise inner-disconnected. Then we have a contradiction from Lemma $\ref{l7}$.

Therefore, we have $| B | = NC2(G)$, and  thus $| V(C) | =2| B |+3=|B\cup B^+|+3$. It means that there are exactly 3 vertices of the cycle $\overrightarrow{C}$ that are not in $B \cup B^+$. We have 3 possibilities for the arrangement of these 3 vertices on the cycle $\overrightarrow{C}$ as follows.
\begin{enumerate}
\item[\textit{I.}] They are in the same interval.

\item[\textit{II.}] They are in two different intervals.

\item[\textit{III.}] They are in three different intervals.
\end{enumerate}

For the sake of contradiction, we will show that all three cases above do not happen.

\subsection{Case I}
Before investigating this case, we present several lemmas stated below.

\begin{lemma}\label{claim1}If $x$ is a vertex on $\overrightarrow{C}$ such that $x^-x^+\in E(G)$ and $x^{+2} \in B$, then $x$ is not adjacent
to any vertex $y \in B^--\{x^-,x,x^+\}$. Analogously, if $x^-x^+\in E(G)$ and $x^{-2} \in B$, then $x$ is not adjacent to any vertex $y \in B^+-\{x^-,x,x^+\}$
\end{lemma}

\begin{figure}\centering
\includegraphics[width=0.75\textwidth]{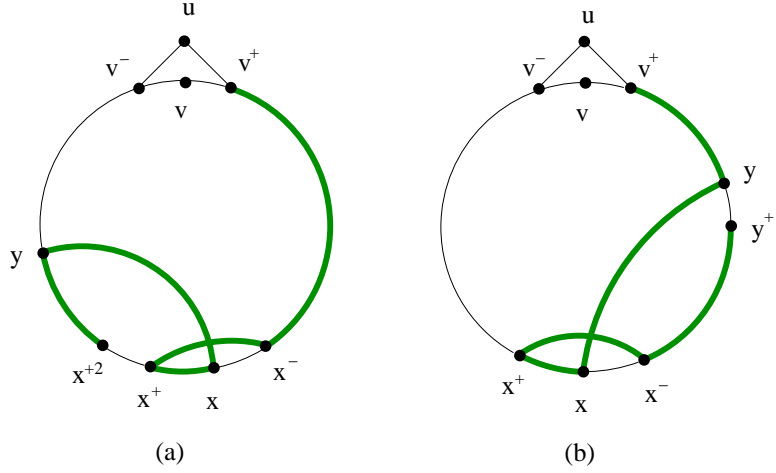}
\caption{Illustrating the proof of Lemma \ref{claim1}}
\label{Fig2}
\end{figure}

\begin{proof}
Note that $v^+$ and $v^-$ are not in $ B^-\cup B^+$, so if $y\in B^-\cup B^{+}$ then $y\notin \{v^-,v^+\}$. We only prove the first statement  (then the second statement follows by reversing the orientation of $\overrightarrow{C}$ as in Remark \ref{trick1}).

Suppose otherwise that $x^-x^+ \in E(G)$ and $x$ is
adjacent to some vertex $y \in B^--\{x^-,x,x^+\}$.  If $y\in x^{+2}\overrightarrow{C}v^{-2}$, then $v^+\overrightarrow{C}x^-x^+xy\overleftarrow{C}x^{+2}$ is a bad path (see Figure \ref{Fig2}(a)). If $y\in v^{+2}\overrightarrow{C}x^{-2}$, then $v^+\overrightarrow{C}yxx^+x^-\overleftarrow{C}y^+$ is a bad path (illustrated in Figure \ref{Fig2}(b)).
\end{proof}

We have a variant of Lemma \ref{claim1}.
\begin{lemma}\label{claim2} Assume that $x,y$ are two distinct vertices on $\overrightarrow{C}$ so that $x^{+}=b_i$ and $y^{-}=b_j$, for some $1\leq j<i\leq m$, and $xy\in E(G)$. Then $x^{-}$ and $y^{+}$ are not adjacent to any inner vertices of 2-intervals on the $C$-path $x^{+2}\overrightarrow{C}y^{-2}$.
\end{lemma}

\begin{proof}
We only need to prove the statement for the vertex $x^{-}$, then the statement for $y^{+}$ follows naturally from reversing orientation trick.

Suppose otherwise that there is $x$ is adjacent to some vertex $a\in B^{+}\cap B^{-}$ on $x^{+2}\overrightarrow{C}y^{-2}$. Apply Hopping Lemma \ref{hopping} to the graph $G$ with cycle $\overrightarrow{C}$ and vertex $u$ as in our setup. We have
 \begin{subequations}
 \begin{align}\label{hoppingeq}
 X_1&=N(u),\\
 Y_1&=N(u)^+\cap N(u)^-\ni v,\\
 X_2&=N(Y_1\cup \{u\})\supseteq B,\\
 Y_2&\supseteq B^+\cap B^-.
 \end{align}
 \end{subequations}
 In particular, $a\in Y_2$, so $x^{-}\in N(a)\subseteq N(Y_2\cup\{u\})=X_3$. By definition, $x\in Y_3$, so $y\in N(x)\subseteq N(Y_3\cup\{u\})=X_4\subseteq X$. However, we already have $y^-\in B\subseteq X_2\subseteq X$, this contradicts the part (b) of the Hopping Lemma \ref{hopping}.
%
\end{proof}

\begin{lemma} \label{claim3}  Assume that $x$ is a vertex on $\overrightarrow{C}$ such that $x$ and $x^+$ are on $ v^{+2}\overrightarrow{C}v^{-2}$. Then there are no two vertices $a$ and $b$ in
$(B^+\cap B^-)-\{x,x^+\}$ such that $xa$ and $x^+b \in E(G)$, where $a$ and $b$ are not necessarily distinct.
\end{lemma}

\begin{proof}
 Suppose otherwise that $xa$ and $x^+b\in E(G)$, for some vertices $a$ and $b$ in $B^+\cap B^-$.  Apply Hopping Lemma \ref{hopping} to the graph $G$.  Similar to Lemma \ref{claim2}, we have (\ref{hoppingeq})--(12d), and $a,b\in Y_2$. Thus by definition, $x,x^+\in N(a)\cup N(b)\subseteq N(Y_2\cup\{u\})=X_3\subseteq X$, contradicting the  part (b) of the Hopping Lemma \ref{hopping}.
 \end{proof}

\medskip
Next, we show that Case I does not happen by contradiction. Suppose otherwise that three vertices of $V(C)-B\cup B^+$ stay in the same interval. Then they are $x_0^{-}$, $x_0$ and $x_0^{+}$, for some vertex $x_0\in V(C)$.
Arguing similarly to the case when $|B|= NC2(G)+1$ in the beginning of Section 3, the cycle $\overrightarrow{C}$ contains one 5-interval and all remaining  intervals have length 2. One readily sees that the 5-interval is $x_0^{-3}\overrightarrow{C}x_0^{+2}=b_{i}\overrightarrow{C}b_{i+1}$, for some $1\leq i\leq m-1$.   From Corollary
\ref{l5c},  $x_0^{-2}$ and $x_0^{+}$ are not adjacent to any inner vertices of 2-intervals.

By Lemma \ref{claim2}, if $x_0^{-2}x_0^{+}\in E(G)$, then $x_0^{-}$ and $x_0$ are not adjacent to any inner vertices of 2-intervals. This implies that the 5-interval $b_i\overrightarrow{C}b_{i+1}$ does not inner-connect to any 2-intervals.  Thus, all the intervals are pairwise inner-disconnected (any two 2-intervals are inner-disconnected), contradicting Lemma \ref{l7}.

Moreover, from Lemma \ref{claim1}, if $x_0^{-2}x_0$ and $x_0^{-}x_0^{+}\in E(G)$, then we have also $x_0^{-}$ and $x_0$ are not adjacent to any inner vertices of 2-intervals, and we get the same contradiction as in the previous paragraph.

Finally, if exactly one of $x_{0}^{-2}x_0$ and $x_0^{-}x_0^{+}$ is in $E(G)$, say $x_{0}^{-2}x_0\in E(G)$ and $x_0^{-}x_0^{+}\notin E(G)$, then $G$ is not 1-tough, a contradiction to our setup \textbf{(S)}.  Indeed, we consider the graph $G-(B\cup \{x_0\})$. Arguing similarly to the proof of Lemma \ref{l7}, the vertices of the set
 \[\{x_0^{-2},x_0^{+},u\}\cup (B^+\cap B^-)\]
 are in distinct components of $G-(B\cup \{x_0\})$. Thus, $G-(B\cup \{x_0\})$ has at least $|B|+2$ components, so $G$ is not 1-tough.

From the contradictions in the three previous paragraphs,  all $x_0^{-2}x_0^{+}$, $x_{0}^{-2}x_0$ and $x_0^{-}x_0^{+}$ are not in $E(G)$. Then at least one of $x_0^{-}$ and $x_0$ is adjacent to an inner vertex of some 2-interval (otherwise, all intervals are pairwise inner-disconnected, then we have a contradiction to Lemma \ref{l7}). On the other hand, by  Lemma \ref{claim3}, exact one of two vertices $x_0^{-}$ and $x_0$ is adjacent to the inner vertex of some 2-interval, say $x_0$. Again, we have the graph $G-(B\cup\{x_0\})$ has at least $|B|+2$ components, so $G$ is not 1-tough, a contradiction. This implies that Case I does not happen.

\subsection{Case II}

We present several supporting lemmas before investigating Case II.

\begin{lemma}\label{l9}  Assume that $a=b_i^+$ and $c=b_j^-$ are two distinct vertices, for some $1\leq i<j\leq m$, so that $ac\in E(G)$.

(1) If there exists a vertex $x\in \{b_{i+1}^-,\dotsc,b_{j-1}^-\}$,  then both $xa^-$ and $xc^+$ are not in $E(G)$. Analogously, we have the same conclusion if there exists a vertex $x\in\{b_{i+1}^+,\dotsc,b_{j-1}^+\}$.

(2) If there exists a vertex $x\in a^+\overrightarrow{C}b^-$ such that $xv^{+2}\in E(G)$, then both $x^+v^+$  and $x^-v^+$ are not in $E(G)$. Analogously, if $xv^{-2} \in E(G)$, then both $x^+v^-$ and $x^-v^-$  are not in $E(G)$.
\end{lemma}

\begin{proof}

\begin{figure}\centering
\includegraphics[width=0.75\textwidth]{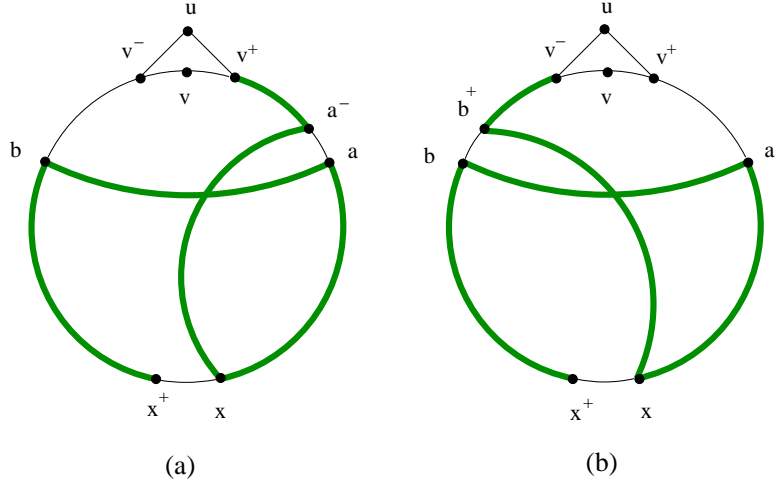}
\caption{Illustrating the proof of Lemma \ref{l9}(1).}
\label{Fig6}
\end{figure}

(1) Assume otherwise that there is a vertex  $x\in \{b_{i+1}^-,\dotsc,b_{j-1}^-\}$ so that $xa^-\in E(G)$ or $xc^+\in E(G)$.

If $xa^- \in E(G)$, then $v^+\overrightarrow{C}a^-x\overleftarrow{C}ac\overleftarrow{C}x^+$ is a bad path (see Figure \ref{Fig6}(a)), which contradicts Proposition \ref{alpha}. If $xc^+\in E(G)$, then $v^-\overleftarrow{C}c^+x\overleftarrow{C}ac\overleftarrow{C}x^+$ is also a bad path (shown in Figure \ref{Fig6}(b)), again from Proposition \ref{alpha} we have a contradiction. This finishes the proof of the first statement.

The case of $x\in \{b_{i+1}^+,\dotsc,b_{j-1}^+\}$ is obtained similarly by reversing orientation trick in Remark \ref{trick1}.
\begin{figure}\centering
\includegraphics[width=0.75\textwidth]{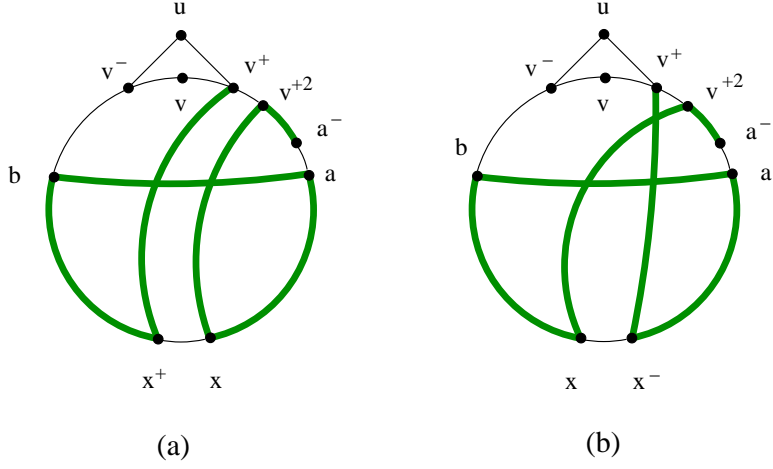}
\caption{Illustrating the proof of Lemma \ref{l9}(2).}
\label{Fig7}
\end{figure}

(2) We only prove the first statement, the second one follows from reversing orientation trick.

Assume otherwise that there is a vertex $x\in a^+\overrightarrow{C}c^-$ so that  $xv^{+2}\in E(G)$.
If $x^+v^+\in E(G)$, then $v^+x^+\overrightarrow{C}ca\overrightarrow{C}xv^{+2}\overrightarrow{C}a^-$ is a bad path (illustrated by Figure \ref{Fig7}(a)). If $x^-v^+\in E(G)$, then $v^+x^-\overleftarrow{C}ac\overleftarrow{C}xv^{+2}\overrightarrow{C}a^-$ is a bad path (shown in Figure \ref{Fig7}(b)). Thus, by Proposition \ref{alpha} again, we have a contradiction. Then the first statement follows.
\end{proof}

\begin{lemma} \label{l10}
Assume that $a=b_i^+$ and $c=b_j^-$ are two distinct vertices, for some $1\leq i<j\leq m$, so that $ac\in E(G)$. If there exists a vertex $x\in N(u)\cap N(v)$ on the $C$-path $a^+\overrightarrow{C}c^-$, then all $x^+v^+$, $x^+v^-$, $x^-v^+$ and $x^-v^-$ are not in $E(G)$.
%
\end{lemma}

\begin{proof}

\begin{figure}\centering
\includegraphics[width=0.75\textwidth]{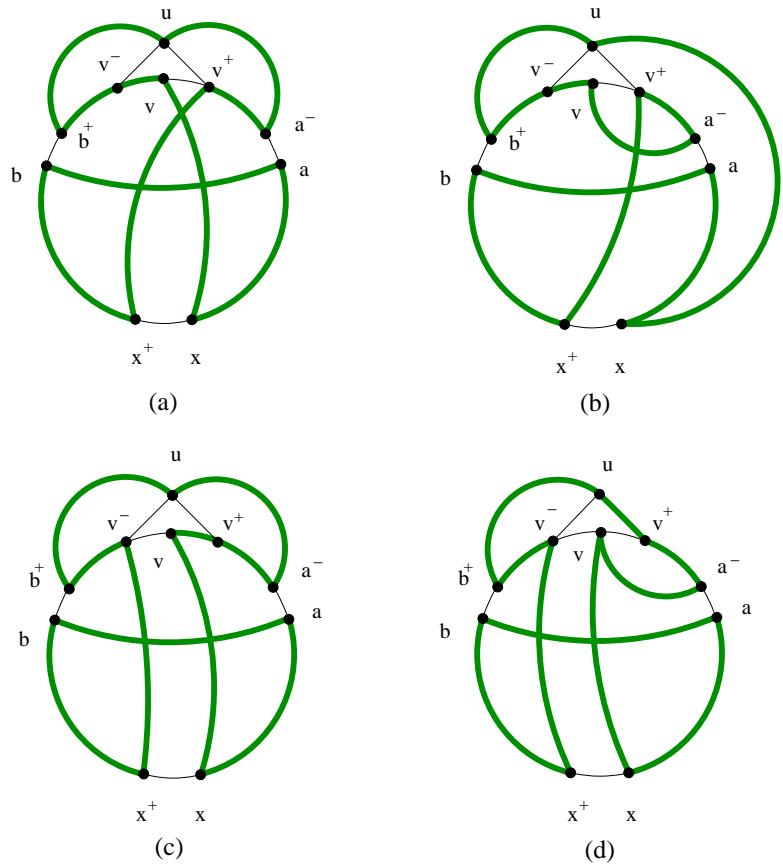}
\caption{Illustrating the proof of Lemma \ref{l10}.}
\label{Fig8}
\end{figure}

Let $x$ be a vertex in $N(u)\cap N(v)\cap a^+\overrightarrow{C}c^-$. If otherwise $x^+v^+\in E(G)$, we construct a cycle $C'$ longer than $C$ as follows.

If $a^-u$ and $c^+u\in E(G)$, then let
$C':=x^+v^+\overrightarrow{C}a^-uc^+\overrightarrow{C}vx\overleftarrow{C}ac\overleftarrow{C}x^+$ (see Figure \ref{Fig8}(a)). If $a^-v,\ c^+u \in E(G)$, then let
$C':=x^+v^+\overrightarrow{C}a^-v\overleftarrow{C}c^+ux\overleftarrow{C}ac\overleftarrow{C}x^+$ (shown in Figure \ref{Fig8}(b)). By interchanging roles trick in Remark \ref{trick2}, we get also a cycle longer than $C$  when $a^-$ and $c^+$ are adjacent to $v$, and when $a^-$ is adjacent to $u$ and $c^+$ is adjacent to $v$.

Thus, in all cases, we get a cycle longer than $C$, a contradiction. This implies that $x^+v^+\notin E(G)$.

If $x^+v^- \in E(G)$, we may construct a cycle $C''$ longer than $C$, then we have a contradiction. By the interchanging roles trick, we only need to consider two following cases. If $a^-u$ and $ c^+u\in E(G)$, then let
$C'':=x^+v^-\overleftarrow{C}c^+ua^-\overleftarrow{C}vx\overleftarrow{C}ac\overleftarrow{C}x^+$ (illustrated in Figure \ref{Fig8}(c)). If $a^-v$ and $c^+u \in E(G)$, then
$C'':=x^+v^-\overleftarrow{C}c^+uv^+\overrightarrow{C}a^-vx\overleftarrow{C}ac\overleftarrow{C}x^+$ (see Figure \ref{Fig8}(d)). This shows that $x^+v^- \notin E(G)$.

Finally, by reversing orientation trick in Remark \ref{trick1}, we get $x^-v^+, x^-v^-\notin E(G)$.
%
\end{proof}

\begin{lemma} \label{l12} Assume that $a=b_p^-$ and $c=b_q^+$ are two distinct vertices, for some $1\leq p<q\leq m$, so that $ac\in E(G)$. Then
\begin{enumerate}
\item[(1)] Either $b_p,b_q\in N(u)\setminus N(v)$ or $b_p,b_q\in N(v)\setminus N(u)$.
\item[(2)] All $b_pv^{-2}$, $b_pv^{+2}$, $b_qv^{-2}$ and $b_qv^{+2}$ are not in $E(G)$.
\end{enumerate}
\end{lemma}

\begin{proof}

\begin{figure}\centering
\includegraphics[width=0.90\textwidth]{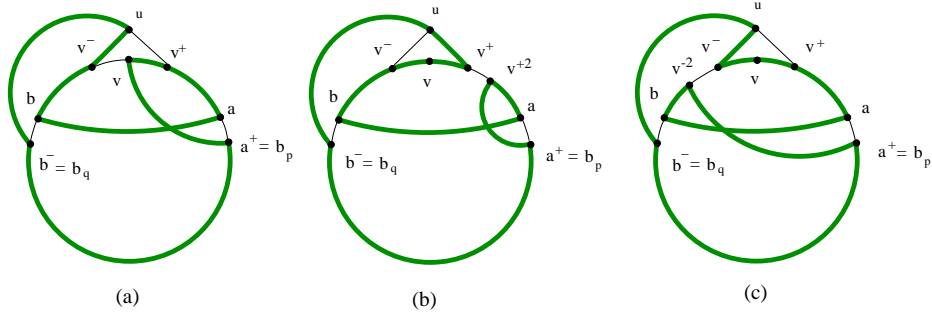}
\caption{Illustrating the proof of Lemma \ref{l12}.}
\label{Fig9}
\end{figure}

(1) Suppose otherwise that one of $b_p$ and $b_q$ is adjacent to $u$ and the other is adjacent to $v$. We only consider the case $b_pv,\ b_qu \in E(G)$ (the other case follows from interchanging roles trick). We have the cycle $C':= vb_p\overrightarrow{C}b_quv^-\overleftarrow{C}ca\overleftarrow{C}v$ is longer than $C$ (see  Figure \ref{Fig9}(a)), a contradiction, and the statement follows.

(2) By part (1), we have $b_pu$ and $b_qu$ are both in $E(G)$, or $b_pv$ and $b_qv$ are both in $E(G)$. Again, we only need to consider the first case (the latter case can be obtained by applying interchanging roles trick).

If $b_pv^{+2} \in E(G)$, the cycle $C'':= b_pv^{+2}\overrightarrow{C}ac\overrightarrow{C}v^+ub_q\overleftarrow{C}b_p$ is longer than $C$, a contradiction  (see Figure \ref{Fig9}(b)).
If $b_pv^{-2} \in E(G)$, the cycle $C'':= b_pv^{-2}\overleftarrow{C}ca\overleftarrow{C}v^-ub_q\overleftarrow{C}b_p$ is longer than $C$, contradicting the choice of $C$  (see Figure \ref{Fig9}(c)). This implies that $b_pv^{+2},\ b_pv^{-2} \notin E(G)$.

By reversing orientation of $\overrightarrow{C}$ as in Remark \ref{trick1},  we get also $b_qv^{+2},\ b_qv^{-2} \notin E(G)$.
\end{proof}

\medskip

Next, we show that Case II does not happen by contradiction. Assume otherwise that three vertices of $V(C)-B\cup B^+$ are falling into two different intervals. Then they are $x_0,$ $x_0^{+}$ and $y_0^{+}$, for some vertices $x_0$ and $ y_0$ on $\overrightarrow{C}$ so that $y_0\notin \{x_0^-,x_0^{+}\}$. Arguing similarly to Case I,
the cycle $\overrightarrow{C}$ consists of a 3-interval ($y_0^-\overrightarrow{C}y_0^{+2}=b_j\overrightarrow{C}b_{j+1}$), a 4-interval ($x_0^{-2}\overrightarrow{C}x_0^{+2}=b_i\overrightarrow{C}b_{i+1}$), and several 2-intervals. Without loss of generality, we assume that $1\leq i<j \leq m$ (the case $i>j$ is obtained from this case by reversing orientation trick). Note that in this case we have $V(C)-B=B^-\cup B^+ \cup\{x_0\}$, $B^--B^+=\{x_0^+,y_0^+\}$, and $B^+-B^-=\{x_0^-,y_0\}$.

\begin{proposition}\label{hopprop}
$x_0$ is not adjacent to any inner vertices of 2-intervals.
\end{proposition}
\begin{proof}
Assume otherwise that $x_0$ is adjacent to some inner vertex $a$ of a 2-interval. Apply Hopping Lemma \ref{hopping} to $G$. Similar to Lemma \ref{claim2}, we have (\ref{hoppingeq})--(12d), and $a\in Y_2$.  We have also $x_0\in N(a)\subseteq N(Y_2\cup\{u\})=X_3$. By definition, $x_0^+,x_0^-\in Y_3$. In particular, by part (d) of the Hopping Lemma, we have $x_0^+x_0^-\notin E(G)$.

If $x_0^+y_0\in E(G)$, then $y_0\in N(x_0^+)\subseteq N(Y_3\cup\{u\})= X_4\subseteq X$. However, we also have $y_0^-\in B\subseteq X_2\subseteq X$. It means that we have two consecutive vertices on $\overrightarrow{C}$ lying in $X$, this contradicts the part (b) of the Hopping Lemma \ref{hopping}. Thus, we have $x_0^+y_0\notin E(G)$. Similarly, we obtain $x_0^-y_0^+\notin E(G)$.

Finally, by Corollary \ref{l5c}, $x_0^+y_0^+$ and $x_0^-y_0$ are not in $E(G)$. Therefore, similar to the second last paragraph of Section 3.1, the graph $G-(B\cup\{x_0\})$ has at least $|B|+2$ components, i.e. $G$ is not 1-tough, a contradiction. Then the proposition follows.
\end{proof}

We also have the following fact about the inner vertices of the 3-interval and the 4-interval.

\begin{proposition}\label{clm1}
$x_0^+x_0^-\notin E(G)$. Moreover, exactly one of  $y_0^+x_0^-$ and $y_0x_0^+$ is in $E(G)$.
\end{proposition}

\begin{proof}
 By Lemma \ref{claim1}, if $x_0^-x_0^+ \in E(G)$, then
$x_0$ is not adjacent to any inner vertices of 2-intervals and the 3-interval. By Corollary \ref{rmk3}, the intervals on $\overrightarrow{C}$ are pairwise inner-disconnected, by Lemma \ref{l7}, we have a contradiction. Therefore, $x_0^-x_0^+\notin E(G)$.

 If both $y_0$ and $y_0^+$ are not adjacent to $x_0^-$ and $x_0^+$, then by arguing similarly to the second last paragraph in Section 3.1, the graph $G- (B\cup \{x_0\})$ has at least $|B|+2$ components. Hence, $G$ is not 1-tough, a contradiction. This implies that at least one of $y_0$ and $y_0^+$ is adjacent to $x_0^-$ or $x_0^+$. Moreover, by Corollary \ref{l5c}, $y_0x_0^-,\ y_0^+x_0^+ \notin E(G)$, so at least one of $y_0^+x_0^-$ and $y_0x_0^+$ is in $E(G)$.

However, if both $y_0^+x_0^-$ and $y_0x_0^+$ are in $E(G)$, then
$b_{i+1}\overrightarrow{C}y_0x_0^+x_0x_0^-y_0^+\overrightarrow{C}v^-$ is a
bad path, which contradicts what we got from Proposition \ref{alpha}. This implies that exactly one of  $y_0^+x_0^-$ and $y_0x_0^+$ is in $E(G)$.
\end{proof}

\begin{proposition}\label{prop1}
$y_0^+x_0^- \notin E(G)$.
\end{proposition}

\begin{proof}
Assume otherwise that $y_0^+x_0^- \in E(G)$. We have the following claims.



%
\begin{claim} \label{clm2}
 $b_{i+1}\equiv b_j$.
\end{claim}
\begin{proof}
Suppose otherwise that $b_{i+1}\not= b_j$,  then $b_{i+1}^+\in B^+ \cap B^-$. Recall that we have $V(C)-B=B^-\cup B^+ \cup\{x_0\}$, $B^--B^+=\{x_0^+,y_0^+\}$, and $B^+-B^-=\{x_0^-,y_0\}$.

By Corollary \ref{l5c} and Proposition \ref{hopprop}, $N(b_{i+1}^+)\subseteq V(C)$ and $N(b_{i+1}^+)\cap (B^+\cup B^-\cup \{x_0\})=\emptyset$. Thus, $N(b_{i+1}^+)\subseteq B$.

By Proposition \ref{clm1}, $y_0x_0^+\notin E(G)$. Since $x_0^+\in B^-$,  Corollary \ref{l5c} implies that $N(x_0^+)\cap (B^-\cup V(G-C))=\emptyset$. Thus, $N(x_0^+)\subseteq B\cup \{x_0\}$.

Finally, from Lemma \ref{l9},  all $b_{i+1}^+b_i$, $b_{i+1}^+b_{j+1}$, $x_0^+b_i$, $x_0^+b_{j+1}$ are not in $E(G)$. Thus, $N(b_{i+1}^+,x_0^+)
\subseteq B\cup\{x_0\}-\{b_i,b_{j+1}\}$. This deduces that $(b_{i+1}^+,x_0^+)$ is a
small pair, contradicting Proposition \ref{small}.
\end{proof}
We have now all $2$-intervals stay in the $C$-path $b_{j+1}\overrightarrow{C}b_i$. The assumption $(S_3)$ of our setup implies that the later $C$-path contains at least four good 2-intervals.

\begin{claim} \label{clm3}
 $b_j\notin N(u)\cap N(v)$.
\end{claim}
\begin{proof}
 Suppose otherwise that $b_j$ is adjacent to both $u$ and $v$. We show that $(x_0^+b_{i+1}^+)$ is still a small pair. Note that we have now $b_{i+1}^+\equiv y_0$.

We have $y_0x_0\notin E(G)$ (otherwise $b_j\overleftarrow{C}x_0y_0y_0^+x_0^-\overleftarrow{C}v^+ $ is a bad path, a contradiction to Proposition \ref{alpha}), and $y_0x_0^+\notin E(G)$ (by Proposition \ref{clm1}). Since $y_0\in B^+$, we have $N(y_0)\subseteq B\cup\{y_0^+\}$.

Moreover, from the proof of Claim \ref{clm2}, $N(x_0^+)\subseteq B\cup \{x_0\}$, and $x_0^+,b_{i+1}^+\equiv y_0$ are not adjacent to $b_{i},b_{j+1}$.

Finally, by Lemma \ref{l10}, all $y_0v^+$, $y_0v^-$, $x_0^+v^+$ and $x_0^+v^-$ are not in $E(G)$. Therefore,
\begin{equation}\label{eqx0y0}
N(y_0,x_0^+)\subseteq B\cup \{y_0^+,x_0\}-\{b_i,b_{j+1},v^+,v^-\}.
\end{equation}
Since  $b_{j+1}\overrightarrow{C}b_i$ contains at least four (good) 2-intervals, we have $b_i\not=v^+$ or $b_{j+1}\not=v^-$. Thus, $|\{b_i,b_{j+1},v^+,v^-\}|\geq 3$, so $|N(y_0,x_0^+)|\leq |B|-1$. This implies that $(y_0,x_0^+)$ is a small pair, a contradiction to Proposition \ref{small}, and the claim follows.
\end{proof}



\begin{claim} \label{clm5}
 If $b_{j+1}\not=v^-$, then $v^{-2}x_0\notin E(G)$. Analogously,  if $b_{i}\not=v^+$, then $v^{+2}x_0\notin E(G)$.
\end{claim}

\begin{proof}
Assume that $b_{j+1}\not= v^-$, then $v^{-2}$ is the inner vertex of some 2-interval. Therefore, the first statement follows from Proposition \ref{hopprop}. Similarly we have the second statement.
%
\end{proof}

\begin{claim} \label{clm6}
If $v^+\not=b_i$, then $v^{+2}b_{i+1}\notin E(G)$. Analogously, if $v^-\not=b_{j+1}$, then $v^{-2}b_{i+1}\notin E(G)$.
\end{claim}

\begin{proof}
We only need to prove the first statement (then the second one follows from the reversing orientation trick).

Suppose otherwise that $v^+\not=b_i$ and $v^{+2}b_{i+1}\in E(G)$. From Lemma \ref{l9}(2), we have $x_0^+v^+,\ y_0v^+ \notin E(G)$.
Similar to (\ref{eqx0y0}), we have
\[N(x_0^+,y_0)\subseteq B\{y_0^+,x_0\}-\{b_i,b_{j+1},v^+\}.\]
 Since $v^+\not=b_i$, we obtain $|\{b_i,b_{j+1},v^+\}|=3$. Thus, $(x_0^+,y_0)$ is a small pair, which contradicts Proposition \ref{small}.
%
\end{proof}

\medskip

As mentioned in the proof of Claim \ref{clm3}, we have $v^+\not=b_i$ or $v^-\not=b_{j+1}$ (due to the assumption $(S_3)$ of our setup on the number of (good) 2-intervals). We have $v^{+2}\in B^+\cap B^-$ or $v^{-2}\in B^+\cap B^-$, respectively, i.e. $\{v^{+2},v^{-2}\}\cap (B^+\cap B^-)\not=\emptyset$.  Assume that $w_1$ is a vertex in the later intersection. Moreover, by Claim \ref{clm3}, we can denote by $w_2$ the vertex in $\{u,v\}$ which is adjacent to $b_{i+1}\equiv b_j$.

By Claims \ref{clm3}, \ref{clm5} and \ref{clm6}, we have $N(w_1,w_2)\subseteq B-\{b_{i+1}\}$. Thus $(w_1,w_2)$ is a small pair, a contradiction.  Hence, the proposition follows.
%
%
\end{proof}

From Propositions \ref{clm1} and  \ref{prop1}, we have $x_0^+y_0\in E(G)$.
\begin{proposition}\label{prop2}
$\{b_{i+1},b_j\}\nsubseteq N(u)$ and $\{b_{i+1},b_j\}\nsubseteq N(v)$.
\end{proposition}
\begin{proof}
%
%
%
Assume otherwise that $b_{i+1}$ and $b_{j}$ are both adjacent to $u$, or both adjacent to $v$, say $b_{i+1}v,\ b_jv \in E(G)$ (the other case follows from interchanging roles trick).

By Lemma \ref{l12}(1),  we have $ub_{i+1},\  ub_j \notin E(G)$. By Lemma \ref{l12}(2), all $v^{+2}b_{i+1}$, $v^{-2}b_{i+1}$, $v^{+2}b_j$, and $v^{-2}b_j$ are not in $E(G)$. 

Moreover, if $b_{j+1}\not= v^-$, then  $v^{-2}$ is the inner vertex of some 2-interval, so is in $B^+\cap B^-$. By Corollary \ref{l5c}, $v^{-2}$ is not adjacent to any vertices in $B^-\cup B^+$, and by Proposition \ref{hopprop}, $v^{-2}x_0\notin E(G)$. Thus, $N(u,v^{-2})\subseteq B-\{b_{i+1},b_j\}$, so $(u,v^{-2})$ is a small pair, a contradiction to Proposition \ref{small}. This implies that $b_{j+1}\equiv v^-$, so $v^{-2}\equiv y_0^+$. Similarly, we have $v^{+}\equiv b_i$.

We have now $N(u,v^{-2})=N(u,y_0^+)\subseteq (B\cup\{y_0\})-\{b_{i+1},b_j\}$. If $b_{i+1}\not=b_j$, then $|N(u,v^{-2})|\leq|B|-1$, so $(u,v^{-2})$ is still a small pair, a contradiction. Thus, $b_{i+1}\equiv b_j$. However, there is only one 2-interval on $\overrightarrow{C}$, contradicting the assumption $(S_3)$ in our setup, which completes the proof of the proposition.
%
%
%
\end{proof}

One readily sees that Lemma \ref{l12}(1) contradicts Proposition \ref{prop2}. This implies that Case II does not happen.

\subsection{Case III}

We prove the following supporting lemma.

\begin{lemma}\label{inva}
Assume that all intervals on $\overrightarrow{C}$ have length 2 or 3.

(a) Assume in addition that $v_0$ is a vertex different from $v$, so that $v_0^+,v_0^-\in N(u)$. Then $(u,v_o,\overrightarrow{C})$ determines a new setup \textbf{($S^{*}$)} of $G$. Moreover, $N(u)\cup N(v _0)=N(u)\cup N(v)=B$.

(b) Assume in addition that $v_0$ is a vertex different from $v$, so that $v_0^+,v_0^-\in N(v)$.  Then $(v,v_0,\overrightarrow{C'})$ determines a new setup \textbf{($S^{**}$)}. Moreover, $N(v)\cup N(v _0)=N(u)\cup N(v)=B$.
%
%
%
\end{lemma}

\begin{proof}
Since part (b) is obtained from part (a) by applying interchanging roles trick, we only present the proof of part (a).

One readily verifies that \textbf{(S*)} satisfies $(S_1)$, $(S_2)$ and $(S_3)$. Next, we show that $B_0:=N(u)\cup N(v _0)=B$.

Since all intervals have length at most 3, the inner vertices of the intervals are in $B^+\cup B^-=V(C)-B$. By Corollary \ref{l5c}, $v_0$ is not adjacent to any vertices of set $B^+\cup B^- \cup V(G-C)$. Thus, $N(v_0)=N(v_0)\cap V(C)\subseteq N(u)\cup N(v)$, so $B_0=N(u)\cup N(v_0)\subseteq N(u)\cup N(v)=B$. Apply the same argument to the setup \textbf{(S*)}, we have also $B\subseteq B_0$. Thus, $B=B_0$.
%
%
\end{proof}

We call a 2-interval having both ends in $N(u)$ or $N(v)$ a \textit{good 2-interval}. Intuitively, Lemma \ref{inva} allows us to ``refine" our setup by relocating  the vertex $v$ to any other inner vertices of good 2-intervals.

\begin{remark}\label{r1}
 Let  $b_p\overrightarrow{C}b_{p+1}$ and $b_q\overrightarrow{C}b_{q+1}$ be  two 3-intervals on $\overrightarrow{C}$, for $1\leq p<q\leq m$, whose inner vertices are $x_1,x_1^+$ and $x_2,x_2^+$, respectively. If these 3-intervals  are inner-connected,  then exactly one of $x_1x_2^+$ and $x_1^+x_2$ is in $E(G)$.
Indeed, from Corollary \ref{l5c}, $x_1x_2$ and $x_1^+x_2^+\notin E(G)$. However, if both $x_1x_2^+$ and $x_1^+x_2$ are in $E(G)$, then
$b_q\overleftarrow{C}x_1^+x_2x_2^+x_1\overleftarrow{C}v^+$ is a
bad path, a contradiction.
\end{remark}


Again, we show that Case III does not occur. Assume otherwise that the three vertices of $V(C)-(B^+\cup B)$ stay in three different intervals. Then they are $x_0^{+}$, $y_0^{+}$ and $z_0^{+}$, for some vertices $x_0$, $y_0$ and $z_0$ on $\overrightarrow{C}$ so that $y_0\notin \{x_0^-,x_0^+\}$ and  $z_0 \notin
\{x_0^-,x_0^+,y_0^-,y_0^+\}$. Arguing similarly to the Case I and Case II,  $\overrightarrow{C}$ consists of three $3$-intervals ($x_0^{-}\overrightarrow{C}x_0^{+2}=b_{i}\overrightarrow{C}b_{i+1}$, $y_0^{-}\overrightarrow{C}y_0^{+2}=b_{j}\overrightarrow{C}b_{j+1}$, and $z_0^{-}\overrightarrow{C}z_0^{+2}=b_{k}\overrightarrow{C}b_{k+1}$, for $1\leq i<j<k\leq m-1$) and several 2-intervals. We notice that all 2-intervals fall into three $C$-paths $b_{i+1}\overrightarrow{C}b_{j}$, $b_{j+1}\overrightarrow{C}b_{k}$, and $b_{k+1}\overrightarrow{C}b_{i}$.

Note that we have in this case $V(C)-B=B^+\cup B^-$, $B^+-B^-=\{x_0^+,y_0^+,z_0^+\}$, and $B^--B^+=\{x_0,y_0,z_0\}$.

We have the following result by arguing similarly to the proofs of Propositions \ref{prop1} and \ref{prop2} in Case II.

\begin{proposition}\label{propn}
 $b_{i}\overrightarrow{C}b_{i+1}$ and $b_{k}\overrightarrow{C}b_{k+1}$ are inner-disconnected.
\end{proposition}

\begin{proof}
Assume otherwise that $b_{i}\overrightarrow{C}b_{i+1}$ inner-connects to $b_{k}\overrightarrow{C}b_{k+1}$. By Remark \ref{r1}, we only need to consider 2 cases as follows.

\medskip
\quad\emph{Case 1.} $x_0z_0^+\in E(G)$.
\medskip


Similar to Claim \ref{clm2} we have
\begin{claim}\label{clm8}
$b_{i+1}\equiv b_j$ and $b_{j+1}\equiv b_k$.
\end{claim}
\begin{proof}
Since the proofs of two statements are essentially the same, we only prove that $b_{i+1}\equiv b_j$.
Assume otherwise that $b_{i+1}\not=b_j$. Similar to the proof of Claim \ref{clm2}, we consider  the pair of vertices $(x_0^+,b_{i+1}^+)$.

We have $x_0^+z_0 \notin E(G)$ (by Remark \ref{r1}), $x_0^+y_0\notin E(G)$ (otherwise we have a bad path $v^+\overrightarrow{C}x_0z^+\overrightarrow{C}y_0x_0^+\overrightarrow{C}b_j$, a contradiction); $x_0^+b_{i},$  $x_0^+b_{k+1},$ $b_{i+1}^+b_{i},$  $b_{i+1}^+b_{k+1} \notin E(G)$ (by Lemma \ref{l9})

Similar to Claim \ref{clm2}, we have
 \begin{equation}
N(x_0^+,b_{i+1}^+)\subseteq B\cup\{x_0\}-\{b_{i},b_{k+1}\}.
\end{equation}
This implies that $(x_0^+,b_{i+1}^+)$ is a small pair, contradicting Proposition \ref{small}. Then the first statement of the claim follows.
\end{proof}

All 2-intervals are now in the $C$-path $b_{k+1}\overrightarrow{C}b_i$. Thus, by the assumption $(S_3)$ in our setup, the later $C$-path contains at least four good 2-intervals.

\begin{claim}\label{clm7}
$v^+\equiv b_{i}$ or $v^-\equiv b_{k+1}$.
\end{claim}

\begin{proof}
Assume otherwise that $v^+\not=b_{i}$ and $v^-\not=b_{k+1}$.

If $b_{i+1}$ is adjacent to both $u$ and $v$,  then by Lemma \ref{l10}, we have $x_0^+$ and $y_0$ are not adjacent to $v^+$ and $v^-$. Moreover, $x_0^+z_0 \notin E(G)$ (by Remark \ref{r1}), $x_0^+y_0\notin E(G)$ (otherwise $v^+\overrightarrow{C}x_0z^+\overrightarrow{C}y_0x_0^+\overrightarrow{C}b_j$ is a bad path, a contradiction); $x_0^+b_{i},$  $x_0^+b_{k+1},$ $y_0b_{i},$  $y_0b_{k+1} \notin E(G)$ (by Lemma \ref{l9}). Therefore,
\begin{equation}\label{eqx+y}
N(x_0^+,y_0)\subseteq B\cup\{x_0,y_0^+,z_0^+\}-\{b_{i},b_{k+1},v^+,v^-\}
\end{equation}
(the vertex $b_{i+1}^+\equiv y_0$ may be adjacent to $y_0^+$ and $z_0^+$).
Since $v^+\not=b_{i}$ and $v^-\not=b_{k+1}$, we have $|\{b_{i},b_{k+1},v^+,v^-\}|=4$, so $(x_0^+,y_0)$ is a small pair, a contradiction to Proposition \ref{small}. Thus, one of $u$ and $v$ is not adjacent to $b_{i+1}$.

If $b_{i+1}v^{+2}, b_{i+1}v^{-2}\in E(G)$, then $x_0^+$ and $b_{i+1}^+$ are not adjacent to $v^+$ and $v^-$ (by Lemma \ref{l9}(2)). It means that we still have (\ref{eqx+y}), so have a contradiction. Hence,  at least one of $v^{+2}$ and $v^{-2}$ is not adjacent to $b_{i+1}$.

Finally, from the facts in the previous two paragraphs, we can denote by $w_1$ (resp., $w_2$)  the vertex in $\{u,v\}$ (resp., in $\{v^{+2},v^{-2}\}$) which is not adjacent to $b_{i+1}$. By definition $N(w_1,w_2)\subseteq B-\{b_{i+1}\}$,  so $(w_1,w_2)$ is a small pair, a contradiction.
\end{proof}

Since  $b_{k+1}\overrightarrow{C}b_i$ contains at least four good 2-intervals, we can find an inner vertex  $v_0$ of a good 2-interval  so that $v_0^+\not=b_i$ and $v_0^-\not=b_{k+1}$. Consider the setup \textbf{(S*)} determined by $(u,v_0,\overrightarrow{C})$ or the setup \textbf{(S**)} determined by $(v,v_0,\overrightarrow{C'})$ as in Lemma \ref{inva}, depending on whether $v_0^+,v_0^-$ are in $N(u)$ or $N(v)$. Apply the argument in Claim \ref{clm7} above to the new setup, we get  $v_0^+\equiv b_{i}$ or $v_0^-\equiv b_{k+1}$, that contradicts the choice of the vertex $v_0$. This contradiction implies $x_0z_0^+\notin E(G)$.

\medskip
\quad\emph{Case 2.} $x_0^+z_0\in E(G)$.

\medskip

By Lemma \ref{l12}, one of $u$ and $v$ is not adjacent to $b_{i+1}$ and $b_{k}$,  say $b_{i+1}u,\ b_ju \notin E(G)$ (the other case follows from interchanging roles trick). Arguing similarly to the proof of Proposition \ref{prop2}. We get $v^+\equiv b_i$ and $v^-\equiv b_{k+1}$.





Then $v^{+2}\equiv x_0$ and $v^{-2}\equiv z_0^+$. By Lemma \ref{l12}(2), $v^{-2}\equiv y_0^+$ is not adjacent to $b_{i+1}$ and $b_{k}$. Moreover, we have $x_0y_0^+ \notin E(G)$ (otherwise we have a bad path $v^+\overrightarrow{C}x_0y_0^+\overrightarrow{C}z_0^+x_0\overrightarrow{C}b_j$, a contradiction). Therefore,
\begin{equation} \label{smallvv+}
N(u,v^{-2})=N(u,y_0^+)\subseteq B\cup\{y_0\}-\{b_{i+1},b_{k}\}.
\end{equation}
However, we have now $(u,y_0^+)$ is a small pair, a contradiction to Proposition \ref{small}. This finishes the proof of the proposition.
%
%
\end{proof}
%
%

Given $b_{i}\overrightarrow{C}b_{i+1}$ and $b_{k}\overrightarrow{C}b_{k+1}$ are inner-disconnected by Proposition \ref{propn}, we get a similar result to Proposition \ref{propn}.
\begin{proposition}\label{propnn}
$b_{i}\overrightarrow{C}b_{i+1}$ and $b_{j}\overrightarrow{C}b_{j+1}$ are inner-disconnected. Analogously, $b_{k}\overrightarrow{C}b_{k+1}$ and $b_{j}\overrightarrow{C}b_{j+1}$ are inner-disconnected.
\end{proposition}
The proof of Proposition \ref{propnn} is omitted, since it is almost identical to the proof of Proposition \ref{propn}.

However, by Propositions \ref{propn} and \ref{propnn}, together with Corollary \ref{rmk3}, all the intervals on $\overrightarrow{C}$ are now pairwise inner-disconnected, so by Lemma \ref{l7} we have a contradiction. This finishes the proof that $\mid V(C)\mid \not=2NC2(G)+3$.

\section{Step 2: Prove $c(G)\not=2NC2(G)+2$.}

Assume that $| V(C) | =2NC2(G)+2$. By Lemma $\ref{l3}$, we have $|V(C)|\geq | B\cup B^+|= | B| +| B^+|=2| B|$. Moreover, $| B| =| N(u)\cup N(v)|
\geq NC2(G)$, so $| B | =NC2(G)$ or $NC2(G)+1$. If  $ | B | =NC2(G)+1$, then $ V(C)= B\cup B^+ $. It implies that the vertices of $B$ divide $C$ into 2-intervals. From Lemma \ref{l7} and Corollary \ref{rmk3}, we have a contradiction. Hence, $| B| =NC2(G)$, so $|V(C)|=|B\cup B^+|+2$. It means that there
are two vertices in $V(C)-B\cup B^+$. We have two cases to distinguish, depending on the arrangement of these vertices on $\overrightarrow{C}$:
\begin{enumerate}
\item[I] They are in the same interval.
\item[II] They are in two different intervals.
\end{enumerate}
By the sake of contradiction, we will show that all of these cases do not happen.

\quad\textit{Case I.}

Assume that the two vertices of $V(C)-B\cup B^+$ stay in the same interval. Then the vertices are $x_0$ and $x_0^+$, for some vertex $x_0\in V(C)$. The cycle
 $\overrightarrow{C}$ consists of one 4-interval ($x_0^{-2}\overrightarrow{C}x_0^{+2}=b_{i}\overrightarrow{C}b_{i+1}$, for some $1\leq i\leq m-1$) and several 2-intervals.

From Lemma \ref{claim1}, if $x_0^-x_0^+\in E(G)$
then $x_0$ is not adjacent to any inner vertices of 2-intervals. Then by Corollary \ref{l5c}, all intervals are pairwise inner-disconnected, and from Lemma \ref{l7} we have a
contradiction. Thus, $x_0^-x_0^+ \notin E(G)$. However, similar to the second last paragraph of Section 3.1, the graph $G- (B\cup \{x_0\})$ has at least $|B|+2$ components, so $G$ is not 1-tough, a contradiction. Thus, the Case I does not happen.

\textit{Case II.}

Assume that the two vertices of $V(C)-B\cup B^+$ fall into two different intervals.  Then the vertices  are $x_0^{+}$ and $z_0^{+}$, where $x_0$ and $z_0$ are some vertices in $V(C)$ and where $z_0\notin\{x_0^-,x_0^+\}$. The cycle  $\overrightarrow{C}$ consists of now two 3-intervals ($x_0^{-}\overrightarrow{C}x_0^{+2}=b_{i}\overrightarrow{C}b_{i+1}$ and $z_0^-\overrightarrow{C}z_0^{+2}=b_{k}Cb_{k+1}$) and several
2-intervals.  Without loss of generality, we assume that $1\leq i<k\leq m-1$.

Similar to Proposition \ref{propn} we have the following fact.

\begin{proposition}\label{propnnn}
$b_{i}\overrightarrow{C}b_{i+1}$ and $b_{k}\overrightarrow{C}b_{k+1}$ are inner-disconnected.
\end{proposition}
Since the proofs Propositions \ref{propn} and \ref{propnnn} are almost identical, we omit the proof of Proposition \ref{propnnn}.

By Propositions \ref{propnnn} and Corollary \ref{rmk3}, all the intervals on $\overrightarrow{C}$ are pairwise inner-disconnected, so by Lemma \ref{l7} we have a contradiction. Then we deduce that Case II does not hold, and also finish the proof that $\mid V(C)\mid \not=2NC2(G)+2$.

\subsection*{\bf Acknowledgements}
Thanks to Professor Vu Dinh Hoa for introducing  to  me this conjecture.


\end{document}